\documentclass[a4paper, 11pt,reqno]{amsart}
\usepackage[top = 1in, bottom = 1in, left=1.2in, right=1.2in]{geometry}

\usepackage[utf8]{inputenc}
\usepackage[activeacute,english]{babel}
\usepackage[T1]{fontenc}
\usepackage{verbatim}
\usepackage{lmodern}
\usepackage{mathtools}
\usepackage{xcolor}
\usepackage{hyperref}
\hypersetup{colorlinks, linkcolor={red!50!black}, citecolor={green!50!black}, urlcolor={blue!50!black}}
\usepackage{graphicx}
\usepackage{caption}
\graphicspath{ {./images/} }
\usepackage{microtype}
\usepackage{bbm}

\usepackage{paralist}
\usepackage[shortlabels]{enumitem}

\usepackage[linesnumbered,ruled,vlined]{algorithm2e}
\usepackage{amsmath}
\usepackage{amssymb,amsfonts}
\usepackage{amsthm}
\usepackage{dsfont}
\usepackage{cite}

\def\NN{\mathbb{N}}
\newcommand{\cS}{\mathcal{S}}
\newcommand{\cB}{\mathcal{B}}
\newcommand{\cI}{\mathcal{I}}
\newcommand{\cJ}{\mathcal{J}}
\newcommand{\cA}{\mathcal{A}}
\newcommand{\cN}{\mathcal{N}}
\newcommand{\cM}{\mathcal{M}}
\newcommand{\cC}{\mathcal{C}}
\newcommand{\cK}{\mathcal{K}}
\newcommand{\cH}{\mathcal{H}}
\newcommand{\cE}{\mathcal{E}}

\newcommand{\stars}{\mathcal{Z}}
\newcommand{\configs}{\mathfrak{B}}

\usepackage{tikz}
\usetikzlibrary{decorations.pathreplacing,calligraphy}
\usetikzlibrary{arrows}
\usetikzlibrary{calc}
 \newcommand{\vertex}[2]{\node[circle, line width=0.7pt, inner sep=#1, draw=#2, fill=#2, minimum size=.1cm]}
   \newcommand{\svertex}[2]{\node[circle, line width=0.3pt, inner sep=#1, draw=#2, fill=#2, minimum size=.03cm]}
 \newcommand{\drawPair}[3]{\draw[#3, thick](#1) -- (#2)}

\theoremstyle{theorem}
\newtheorem{theorem}{\textbf{Theorem}}[section]
\newtheorem{corollary}[theorem]{\textbf{Corollary}}
\newtheorem{lemma}[theorem]{\textbf{Lemma}}
\newtheorem{claim}[theorem]{\textbf{Claim}}
\newtheorem{proposition}[theorem]{\textbf{Proposition}}
\theoremstyle{definition}
\newtheorem{definition}[theorem]{\textbf{Definition}}
\theoremstyle{fact}

\definecolor{darkgreen}{rgb}{0.0, 0.7, 0.0}

\newcommand{\extset}{\mathrm{Ext}}
\newcommand{\drawThreeHyperedge}[5]{\draw[#4, line width=1pt]\convexpath{#1,#2,#3}{#5}}
\newcommand{\drawFourHyperedge}[6]{\draw[#5, line width=1pt]\convexpath{#1,#2,#3,#4}{#6}}

\newcommand{\convexpath}[2]{
   [   
   create hullcoords/.code={
     \global\edef\namelist{#1}
     \foreach [count=\counter] \nodename in \namelist {
       \global\edef\numberofnodes{\counter}
       \coordinate (hullcoord\counter) at (\nodename);
     }
     \coordinate (hullcoord0) at (hullcoord\numberofnodes);
     \pgfmathtruncatemacro\lastnumber{\numberofnodes+1}
     \coordinate (hullcoord\lastnumber) at (hullcoord1);
   },
   create hullcoords
   ]
   ($(hullcoord1)!#2!-90:(hullcoord0)$)
   \foreach [
   evaluate=\currentnode as \previousnode using \currentnode-1,
   evaluate=\currentnode as \nextnode using \currentnode+1
   ] \currentnode in {1,...,\numberofnodes} {
     let \p1 = ($(hullcoord\currentnode) - (hullcoord\previousnode)$),
     \n1 = {atan2(\y1,\x1) + 90},
     \p2 = ($(hullcoord\nextnode) - (hullcoord\currentnode)$),
     \n2 = {atan2(\y2,\x2) + 90},
     \n{delta} = {Mod(\n2-\n1,360) - 360}
     in 
     {arc [start angle=\n1, delta angle=\n{delta}, radius=#2]}
     -- ($(hullcoord\nextnode)!#2!-90:(hullcoord\currentnode)$) 
                                            }
}

\begin{document}

\title[High-order bootstrap percolation in hypergraphs]{High-order bootstrap percolation in hypergraphs}

\author[O. Cooley, J. Zalla]{Oliver Cooley$^{*}$,
Julian Zalla$^{*}$}

\renewcommand{\thefootnote}{\fnsymbol{footnote}}

\footnotetext[1]{Supported by Austrian Science Fund (FWF): I3747, W1230, \texttt{\{cooley,zalla\}@math.tugraz.at},
Institute of Discrete Mathematics, Graz University of Technology, Steyrergasse 30, 8010 Graz, Austria.}

\renewcommand{\thefootnote}{\arabic{footnote}}
\begin{abstract}
Motivated by the bootstrap percolation process for graphs, we define
a new, high-order generalisation to $k$-uniform hypergraphs,
in which we infect $j$-sets of vertices for some integer $1\le j \le k-1$.
We investigate the smallest possible size of an initially infected set which ultimately percolates
and determine the exact size in almost all cases of $k$ and $j$. 
\end{abstract}

\maketitle

\section{Introduction and Main results} \label{sec:intro}

\subsection{Motivation}
Bootstrap percolation was first proposed by Chalupa, Leath and Reich~\cite{ChalupaLeathReich79} in 1979 to study diluted magnetic systems using a special kind of lattice. 
Since then it has been successfully used to describe a plethora of phenomena ranging from virus infections in human populations~\cite{DreyerRoberts09} to living neural networks~\cite{Amini10}, 
belief propagation in social networks~\cite{NguyenZheng12, HelbingBook} and many more. 

Informally, given a graph $G$, a natural number $r$ and a set $A$ of initially infected vertices, in each step of the bootstrap percolation process
(also known as $r$-neighbour bootstrap percolation)
we infect previously uninfected vertices that are incident to at least $r$ currently infected vertices.
We call $r$ the \emph{infection threshold} and note that once a vertex becomes infected,
it stays infected forever. If there exists a time at which all vertices are infected, we call the initial set $A$ \emph{contagious}.
This may be seen as a generalisation of the notion of connectedness, since for $r=1$ the final infected set contains
simply those vertices in the components of the initially infected set.

The graph $G$ may be deterministic, such as
the $d$-dimensional grid $[n]^d$ (see~\cite{AizenmannLebowitz88,BaloghBollobasMorris09,BBDCM10,BaloghBollobasDuminilCopinMorris12}),
or infinite trees~\cite{BaloghPeresPete06,BGHJP14}.
Alternatively, $G$ may be chosen randomly, for example from the $G(n,p)$ or $G(n,m)$ models
(see~\cite{JansonLuczakTurovaVallier12,FeigeKrivelevichReichman17,AngelKolesnik18}), or random graphs with a given degree sequence~\cite{Janson09}.
Typical questions include how large the size of the final infected set is or how many steps the process runs for until completion.
The answers to these questions are in general heavily dependent on the choice of the initially infected set $A$,
whose vertices may be selected independently at random or, in a more recent approach,
one can try to find the smallest contagious set $A$ in $G$~\cite{FeigeKrivelevichReichman17,AngelKolesnik18}.
A closely related process is graph bootstrap percolation~\cite{Bollobas68,BaloghBollobasDuminilCopinMorris12,AngelKolesnik18},
originally proposed by Bollob{\'a}s in $1968$.

Bootstrap percolation has also been studied in (random) $k$-uniform hypergraphs in~\cite{KangKochMakai17}, 
where an infection process on the \emph{vertices} of a random hypergraph was studied;
by contrast,
inspired by recent work on high-order connectedness and percolation processes in hypergraphs (e.g., \cite{BCKK17,CooleyKangKoch16,CooleyKangKoch18}), 
which evolve on \emph{sets of vertices} rather than the vertices themselves,
for any $1\le j\le k-1$ we introduce a process on the \emph{$j$-sets} (of vertices) of the $k$-uniform hypergraph, which has not been considered previously in the literature.

\subsection{$j$-set bootstrap percolation process}
In this section we introduce
a natural generalisation of the standard graph process to $k$-uniform hypergraphs
which evolves on the $j$-sets of vertices, where $1 \le j \le k-1$. For an integer $m\in \NN$ and a set $S$,
we denote $[m]\coloneqq \{1,2,\ldots,m\}$ and write $\binom{S}{m}$ for the set of $m$-element subsets of $S$.

\begin{definition}\label{def:infection}
Given integers $j,k,r,n \in \NN$ with $1\le j \le k-1$, 
a $k$-uniform hypergraph $\cH$ on vertex set $[n]$ with edge set $\cE \subset \binom{[n]}{k}$
and $\cA_0\subset\binom{[n]}{j}$, for each $t \in \NN$, we recursively define
$\cA_t\coloneqq \cA_{t-1}\cup \cN_t$, where $\cN_t$ is the set of all $J\in\binom{[n]}{j}\setminus \cA_{t-1}$ such that the following holds:
\begin{enumerate}[(\textasteriskcentered)]
\item There exist distinct edges $K_1,\ldots,K_r\in \cE$ and distinct $j$-sets $J_1,\ldots,J_r\in \cA_{t-1}$ with $(J_i\cup J)\subset K_i$ for all $i\in[r]$.
\end{enumerate}
\end{definition}
We call $\cA_0$ the initially infected set and $\cA_t$ the set of infected $j$-sets at time $t$.
We refer to this procedure as the \emph{$(r,\cA_0)$-infection process} where $r$ is the \emph{infection threshold}. Note that the hypergraph $\cH$ 
and the parameters $k$ and $j$ are implicit in this notation and will always be clear from the context.
Note also that if $k=2$ and $j=1$, this is identical to the standard bootstrap percolation process for graphs.
The only case which has been previously studied for hypergraphs is the case $j=1$~\cite{KangKochMakai17}.

Once a $j$-set becomes infected it will stay infected forever and we define $\tau\coloneqq\min\{t\in\NN:\cA_t=\cA_{t-1}\}$ to be the earliest time 
when the process does not infect any new $j$-sets. We note that $\tau$ exists because the sets $A_t$ are monotonically increasing with $t$. 
We say $\cA_0\subset\binom{[n]}{j}$ is \emph{contagious} if the final infected set is $\cA_{\tau}=\binom{[n]}{j}.$ 

For any $i \in \NN$, we call a set of $i$-sets an \emph{$i$-configuration} (so in particular, the sets $A_t$ above are $j$-configurations).

\subsection{Main results}

The main goal of this paper is to initiate the study of $j$-set bootstrap percolation by considering perhaps the most natural and fundamental question of all:
What is the size and structure of the smallest possible contagious initial $j$-configuration?
Therefore we will focus on the complete $k$-uniform hypergraph in this paper -- denote by $\cK_n^k$ the complete $k$-uniform hypergraph on $[n]$.

\begin{definition}
Let $\ell_n(k,j,r)$ denote the minimum size
(i.e., number of $j$-sets) over all contagious $j$-configurations in $\cK_n^k$ with infection threshold $r$.
\end{definition}

We note that for $j=1$, i.e., for all cases previously studied, trivially any set of $r$ vertices will percolate just one step,
so $\ell_n(k,1,r)=r$ (except for some degenerate cases when $n$ is too small).
Indeed, this also holds for all $j\le k/2$. Thus it is only when $j$ becomes larger that this problem becomes interesting,
a typical example of the richer phenomena exhibited by high-order structures.
Our main result is the following.

\begin{theorem}\label{thm:maintheorem}
Let $k,j,r\in\NN$. There exists $n_0\coloneqq n_0(k,j,r)\in\NN$ such that for all $n\ge n_0$ we have the following.
\begin{enumerate}[(i)]
\item    \label{item:smallj}  If $1\le j< k-1$, \hspace{0.5cm}  $\ell_n(k,j,r) =r $;
\vspace{0.2cm}
\item    \label{item:exactj} \hspace{3.43cm} $\ell_n(3,2,r)  =\frac{1}{4}\left((r+1)^2-\mathbbm{1}_{r\in 2\NN}\right)$;
\vspace{0.2cm}
\item    \label{item:largej} 
    if $3\le j=k-1$, \hspace{0.55cm} $\ell_n(k,j,r)  \le \sum_{i=1}^{r}\ell_n(k-1,j-1,i)$.
\end{enumerate}
\end{theorem}

Note that while we have not given an explicit value for $n_0$ since this would involve some annoying technical considerations,
one could check that the conditions we need are $n_0\ge \max\{k+r-1,2r+1\}$, and also
that $n_0$ is large enough that the problem is not degenerate because $\binom{n}{j}$ is smaller than the value given for $\ell$.

Observe that by recursively applying~\ref{item:largej} and substituting in~\ref{item:exactj},
we can obtain an explicit upper bound on $\ell_n(k,j,r)$ for the case when $3\le j=k-1$,
which is a polynomial in $r$ of degree $k-1$. We present and discuss this upper bound explicitly
in Section~\ref{sec:remarks}, where we also observe that in general this upper bound is
not best possible, even up to a constant factor.

\subsection{Key proof techniques}
An essential tool in our arguments is the idea of \emph{reduction}, in which we delete certain vertices
and reduce to a similar (and in some cases equivalent) infection process with a lower infection threshold (Lemma~\ref{lem:reductionlemma}).
Another tool that is crucial to derive the recursive inequality in~\ref{item:largej}
is the \emph{augmentation} of a configuration (see Definition~\ref{def:augmentation})
in which a vertex is added to each set of the configuration; this may be viewed as the reverse
operation of considering the ``link hypergraph'' of a vertex, a common feature in hypergraph theory.

\subsection{Paper overview}
The remainder of the paper is organised as follows.

In Section~\ref{sec:prelim} we introduce
some preliminary concepts and basic results that will be used in the remainder of the paper.
In Section~\ref{sec:nontight} we
prove the first statement of Theorem~\ref{thm:maintheorem} (the \emph{non-tight} case, when $j \le k-2$).

Subsequently, in Section~\ref{sec:prelimtight}
we collect some further auxiliary results that are needed specifically for the tight cases when $j=k-1$.
In Section~\ref{sec:specialtight}, we then prove~\ref{item:exactj}, i.e., the second statement of Theorem~\ref{thm:maintheorem},
while in 
Section~\ref{sec:generaltight} we prove \ref{item:largej}.
For clarity, in Section~\ref{sec:proofofmainthm} we complete the proof of Theorem~\ref{thm:maintheorem}
by indicating precisely where each case appears as a statement we have previously proved.
We conclude the paper in Section~\ref{sec:remarks} with some open questions, in particular
observing that the bound in~\ref{item:largej} is not tight in general.

\section{Preliminaries}\label{sec:prelim}
Throughout this paper we fix $j,k\in\NN$ with $j\le k-1$ and will implicitly assume that $n\ge n_0(k,j,r)$ is large enough. Let $\cH$ be an $s$-uniform hypergraph on $[n]$ and let $R\subset [n]$. 
Then we define $\cH-R$ to be the $s$-uniform hypergraph
induced by $\cH$ on $[n]\backslash R$, i.e., with vertex set $[n]\setminus R$ and whose edges are precisely those
edges of $\cH$ which do not contain a vertex of $R$.
An $i$-configuration $\cI$ may also be the edge set of an $i$-uniform hypergraph, and we often
identify an $i$-configuration with the corresponding hypergraph (with the smallest possible vertex set).
In particular, we will use the $\cI-R$ notation for $i$-configurations.

Whenever we talk of a percolation process, the underlying hypergraph will 
always be $\cH=\cK_n^k$ unless otherwise is specified.

In our arguments, it will be particularly convenient to keep track of those parts of the hypergraph where everything possible
has already been infected, which motivates the following definitions. 

\begin{definition}\label{def:jokers}
Given two sets $S,V$ of vertices with $S\subset V$, we define
\begin{align*}
S^* & \coloneqq \left\{J\in\binom{V}{j}:S\subset J\right\}.
\end{align*}
If $S=\{v\}$ is a singleton we will write $v^*\coloneqq\{v\}^*$ for ease of notation.
\end{definition}
In other words, $S^*$ is the $j$-configuration consisting of every $j$-set which contains $S$.
Note that $V$ is implicit in this notation, but will always be clear from the context.
\begin{definition}
Let $S$ be a set of $s<j$ vertices and let $\cJ$ be a $j$-configuration.
We say $S$ is a \emph{joker $s$-set for $\cJ$} if $S^* \subset \cJ$. In particular,
if $s=1$ we call the vertex of $S$ a \emph{joker vertex} (or simply \emph{joker}) for $\cJ$.
If $\cJ=\cA_t$ and $\cA_t$ is clear from the context we simply say that $S$ is a \emph{joker $s$-set at time $t$}.
\end{definition}

If we have $s$ joker vertices, they can always be used in the percolation process,
effectively reducing the analysis of the remaining process to one with infection threshold $r-s$ in the remaining hypergraph.
The following lemma makes this precise.

\begin{lemma}\label{lem:reductionlemma}
Let $\cB_0$ be a $j$-configuration, let $S=\{v_1,\ldots,v_s\}$ be a set of $s$ jokers for $\cB_0$ and let
$\cC_0\coloneqq \cB_0- S$. Let $\cB_t$ be the set of infected $j$-sets at time $t$ in the $(r,\cB_0)$-infection process in $\cK_n^k$ and 
let $\cC_t$ be the set of infected $j$-sets at time $t$ in the $(r-s,\cC_0)$-infection process in $\cK_n^k-S$.
Then $\cC_t\subset \cB_t- S$ for all $t\in\NN$.

Furthermore, if $j=k-1$ then in fact $\cC_t=\cB_t- S$.
\end{lemma}
\begin{proof}
We prove the statement by induction on $t$. For $t=0$ we have $\cC_0=\cB_0- S$ by definition, so suppose the statement holds for $t-1$. 
Let $\cN_t\coloneqq \cC_t\backslash \cC_{t-1}$ and $\cM_t\coloneqq \cB_t\backslash \cB_{t-1}$. By our induction hypothesis we have that 
\begin{equation*}
  \cC_t=(\cC_{t-1}\cup \cN_t)\subset \big((\cB_{t-1}- S)\cup \cN_t\big)
\end{equation*}
and we aim to show that this is a subset of $\cB_t-S = (\cB_{t-1}\cup \cM_t)-S$.
It therefore suffices to show that 
\[
\cN_t\subset(\cB_{t-1}\cup \cM_t)-S.
\]

Let $J\in \cN_t$. Note that because the $j$-sets of $\cN_t$ are
infected in a process within $\cK_n^k -S$, we certainly have $J \cap S = \emptyset$, so it suffices to show that $J \in \cB_{t-1}\cup \cM_t$.
Suppose that $J\not\in \cB_{t-1}$, otherwise we are done. 
Let $K_1,\ldots,K_{r-s}$ and $J_1,\ldots,J_{r-s}\in \cC_{t-1}$ be the edges and $j$-sets which cause $J$ to become infected in $\cK_n^k-S$.
Let us fix a $(k-j-1)$-set $Q\subset V\backslash(J\cup S)$ and a $(j-1)$-set $J'\subset J$. We set $K_i\coloneqq J\cup Q\cup \{v_{i-r+s}\}$
and $J_i=\{v_{i-r+s}\}\cup J'$ for $r-s+1\le i\le r$ (see Figure~\ref{fig:KiConstruction}).
We claim that the $k$-sets $K_1,\ldots,K_r$ and $j$-sets $J_1,\ldots,J_r$ cause $J$ to become infected in $\cK_n^k$.

\begin{center}
\begin{figure}[h]
\begin{tikzpicture}[scale=1]
\vertex{.005cm}{black} at (3.7,3) (33) {};
\vertex{.005cm}{black} at (4.7,3) (43) {};
\vertex{.005cm}{black} at (3.3,4) (3345) {};
\vertex{.005cm}{black} at (3.8,4) (3845) {};
\vertex{.005cm}{black} at (4.3,4) (4345) {};
\vertex{.005cm}{black} at (4.8,4) (4845) {};
\vertex{.005cm}{black} at (4,5.7) (46) {};
\foreach \x in {4,...,10}
\svertex{.0005cm}{black} at (\x/3,5.7) {};
\foreach \x in {14,...,20}
\svertex{.0005cm}{black} at (\x/3,5.7) {};
\vertex{.005cm}{black} at (4/3,5.7) {};
\node at (4/3,6.1) {\small \color{black} $v_1$};
\vertex{.005cm}{black} at (20/3,5.7) {};
\node at (20/3,6.1) {\small \color{black} $v_s$};
\draw (4,5.7) ellipse (3cm and 0.35cm);
\draw (4.1,4) ellipse (1.2cm and 0.43cm);
\draw (4.2,3) ellipse (0.8cm and 0.3cm);
\drawThreeHyperedge{3845}{46}{4845}{black}{.2cm};
\draw [rounded corners=13pt]
(2,3)--(2,5.2)--(3.2,5.2)--(4,6.2)--(4.8,5.2)--(6,5.2)--(6,2.5)--(2,2.5)--(2,3);
\node at (4,6.2) {\small \color{black} $v_{i-r+s}$};
\node at (0.6,5.6) {\small \color{black} $S$};
\node at (4.9,4.8) {\small \color{black} $J_i$};
\node at (2.5,4) {\small \color{black} $J$};
\node at (3,3) {\small \color{black} $Q$};
\node at (6.3,3.8) {\small \color{black} $K_i$};
\end{tikzpicture}
\caption{
The construction of $J_i$ and $K_i$ for $i\ge r-s+1$.
}\label{fig:KiConstruction}
\end{figure}
\end{center}

First observe that all $K_i$ are distinct: For $i\le r-s$ they are distinct since by assumption they are the $k$-sets causing $J$ to become infected 
in $\cK_n^k-S$ and for $i\ge r-s+1$ all $v_{i-r+s}$ are distinct and thus all $K_i$ are also distinct. Finally, if $i\le r-s< i'$ then $K_i\not =K_{i'}$ 
because $K_i\cap S=\emptyset\not =\{v_{i'-r+s}\}=K_{i'}\cap S$. An identical argument also shows that the $J_i$ are distinct, and clearly $(J_i\cup J)\subset K_i$ for all $1\le i\le r$.
Finally observe that $J_i\in \cB_{t-1}$: For $i \le r-s$ this is because $J_i \in \cC_{t-1} \subset \cB_{t-1}-S \subset \cB_{t-1}$ by the induction hypothesis,
while for $i>r-s$ the vertex $v_{i-r+s}\in J_i$
is a joker for $\cB_0\subset \cB_{t-1}$. This shows that $J\in \cM_t$, as required. This proves the first statement of the lemma.

Now suppose that $j=k-1$, and we will show the other inclusion, $\cB_t-S\subset \cC_t$.
Note that by 
the induction hypothesis, $\cC_t=\cC_{t-1}\cup \cN_t=(\cB_{t-1}-S)\cup \cN_t$,
while also the fact that $\cM_t=\cM_t-S$ implies that
$ \cB_t-S = (\cB_{t-1}\cup \cM_t) - S = (\cB_{t-1}-S)\cup \cM_t$. Therefore it is enough to show that
$\cM_t\subset \cN_t$.

Suppose that $J\in \cM_t$, so in particular $J \cap S = \emptyset$. Then
there exist distinct edges $K_1,\ldots,K_{r}\in \binom{[n]}{k}$ and distinct $j$-sets $J_1,\ldots,J_{r}\in \cB_{t-1}$ such that $(J_i\cup J)\subset K_i$ 
for all $i\in[r]$. Observe that $|K_i\backslash J|=k-j=1$ and let $x_i$ be the unique element of $K_i\backslash J$.
Since all the $K_i$ are distinct, so are all the $x_i$. Thus at most $s$ of the $x_i$ lie in $S$, and therefore at least $r-s$ of the $K_i$ are such that $K_i\cap S=\emptyset$.
Without loss of generality we assume that $K_1,\ldots,K_{r-s}$ are disjoint from $S$ (see Figure~\ref{fig:KiConstruction2}).

\begin{center}
\begin{figure}[h]
\begin{tikzpicture}[scale=1]
\vertex{.005cm}{black} at (3.3,4) (3345) {};
\vertex{.005cm}{black} at (3.8,4) (3845) {};
\vertex{.005cm}{black} at (4.3,4) (4345) {};
\vertex{.005cm}{black} at (4.8,4) (4845) {};
\vertex{.005cm}{black} at (4,5.7) (46) {};
\foreach \x in {4,...,10}
\svertex{.0005cm}{black} at (\x/3,5.7) {};
\foreach \x in {14,...,20}
\svertex{.0005cm}{black} at (\x/3,5.7) {};
\vertex{.005cm}{black} at (4/3,5.7) {};
\node at (4/3,6.1) {\small \color{black} $x_{1}$};
\vertex{.005cm}{black} at (20/3,5.7) {};
\node at (20/3,6.1) {\small \color{black} $x_{r-s}$};
\draw (4.1,4) ellipse (1.2cm and 0.43cm);
\draw (4.05,2.5) ellipse (0.8cm and 0.3cm);
\drawThreeHyperedge{3845}{46}{4845}{black}{.2cm};
\draw [rounded corners=13pt]
(2,3.8)--(2,5.2)--(3.2,5.2)--(4,6.2)--(4.8,5.2)--(6,5.2)--(6,3.3)--(2,3.3)--(2,3.8);
\node at (4,6.2) {\small \color{black} $x_i$};
\node at (3,2.5) {\small \color{black} $S$};
\node at (4.9,4.8) {\small \color{black} $J_i$};
\node at (2.5,4) {\small \color{black} $J$};
\node at (6.3,3.8) {\small \color{black} $K_i$};
\vertex{0.005cm}{black} at (3.55,2.5) (s1) {};
\vertex{0.005cm}{black} at (4.05,2.5) (s2) {};
\vertex{0.005cm}{black} at (4.55,2.5) (s3) {};
\end{tikzpicture}
\caption{
The construction of $J_i$ and $K_i$ for $1\le i\le r-s$.
}\label{fig:KiConstruction2}
\end{figure}
\end{center}

We have that all $K_1,\ldots,K_{r-s}$ and $J_1,\ldots,J_{r-s}$ are distinct and satisfy $(J_i\cup J)\subset K_i$.
Furthermore $J_i\subset K_i\subset[n]\backslash S$ and thus, by our induction hypothesis, $J_i\in \cB_{t-1}-S=\cC_{t-1}$ for $i\in [r-s]$.
It follows that $J\in \cN_t$, and since $J\in \cM_t$ was chosen arbitrarily this shows that $\cM_t\subset \cN_t$, as required.
\end{proof}

\section{Non-tight case: $j\le k-2$}\label{sec:nontight}
In the non-tight case where $j\le k-2$, we show that the size of the minimal contagious initial set is precisely $r$,
i.e., the statement of Theorem~\ref{thm:maintheorem}\ref{item:smallj}. It is clear that the size of a contagious set is always at least $r$,
and to prove the corresponding upper bound we will define a $j$-configuration of size $r$ that does in fact percolate in $\cK_n^k$.
\begin{definition}\label{def:mjstar}
Let $m$ be an integer satisfying $0\le m \le j-1$. An \emph{$(m,j)$-star $\cS$ of size $s$} consists of $s$ distinct $j$-sets
every two of which intersect precisely in a common $m$-set $M$.
We call $M$ the \emph{centre} of the $(m,j)$-star. If $m,j$ are clear from the context we call an $(m,j)$-star simply a \emph{star}.
\end{definition}
\begin{center}
\begin{figure}[h]
\begin{tikzpicture}[scale=1]
\vertex{.005cm}{black} at (0,0) (00) {};
\vertex{.005cm}{black} at (1,0) (10) {};
\vertex{.005cm}{black} at (2,0) (20) {};
\vertex{.005cm}{black} at (0.2,1.8) (021) {};
\vertex{.005cm}{black} at (1.8,1.8) (181) {};
\vertex{.005cm}{black} at (2.8,1.1) (2505) {};
\vertex{.005cm}{black} at (-0.8,1.1) (0505) {};
\vertex{.005cm}{black} at (-0.5,0.8) {};
\vertex{.005cm}{black} at (0.4,1.4)  {};
\vertex{.005cm}{black} at (1.6,1.4)  {};
\vertex{.005cm}{black} at (2.5,0.8)  {};

\drawThreeHyperedge{00}{0505}{20}{blue}{.2cm};
\drawThreeHyperedge{00}{021}{20}{darkgreen}{.2cm};
\drawThreeHyperedge{00}{181}{20}{red}{.2cm};
\drawThreeHyperedge{00}{2505}{20}{orange}{.2cm};

\draw[dashed] (1,0) ellipse (1.4cm and 0.43cm);
\node at (1,-0.7) {\small $M$};

\end{tikzpicture}
\caption{
A $(3,5)$-star of size $4$.
}\label{fig:Kiconstruction}
\end{figure}
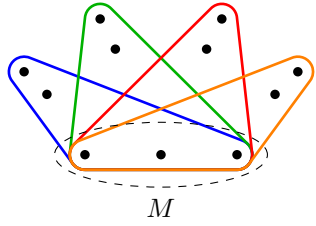
\end{center}

Note that a $(1,2)$-star is simply the standard graph notion of a star.
We will show that, under some parameter assumptions, an $(m,j)$-star of size $r$ percolates.
We begin by showing that, under appropriate conditions, if a star of size $r$ is infected,
then its centre becomes a joker $m$-set in the next step.
\begin{proposition}\label{prop:getfullstar}
Suppose that $2j-k\le m\le j-1$, and let $M$ be the centre of an $(m,j)$-star $\cS$ of size $r$. If $\cS\subset \cA_t$ then $M^*\subset \cA_{t+1}$.
\end{proposition}
\begin{proof}
Let $J$ be an arbitrary $j$-set containing $M$, and we aim to show that $J\in \cA_{t+1}$.
If $J\in \cS\subset \cA_t \subset \cA_{t+1}$ we are done, so assume that $J\not\in \cS$. Let $J_1,\ldots,J_r$ be the $j$-sets of $\cS$. 
For each $i\in[r]$ let $K_i$ be a $k$-set containing $J\cup J_i$ and a further $k-|J\cup J_i|$ vertices from $[n]\setminus (V(\cS) \cup J)$.
Note that it is possible to choose such a $K_i$ since $k-|J\cup J_i| \ge k-(2j-m) \ge 0$.

We need to show that all $K_i$ are distinct for $i\in[r]$ and for convenience we show that $K_1$ and $K_2$ are distinct since the full claim 
then follows by symmetry. Observe that $K_i\cap V(\cS)=J\cup J_i$ for $i=1,2$, so it is enough to show that $J\cup J_1$ and $J\cup J_2$ are distinct. 
We have $(J\cup J_1=J\cup J_2)\Longleftrightarrow (J_1\triangle J_2\subset J)$. Note that $J_1\cap J_2=M\subset J$ since $\cS$ is an $(m,j)$-star. 
Therefore if $J_1 \Delta J_2 \subset J$, we have $|J| \ge m + 2(j-m) >j$, a contradiction. 

Therefore the $J\cup J_i$ are all distinct and so we have distinct $k$-sets $K_1,\ldots,K_r$ and distinct $j$-sets $J_1,\ldots,J_r \in \cA_t$ such that $(J \cup J_i) \subset K_i$,
which is precisely the condition required to guarantee that $J\in \cA_{t+1}$. Since $J$ was an arbitrary $j$-set containing $M$,
the statement follows.
\end{proof}

We next claim that if we have a joker $m$-set, other $m$-sets which are ``close by'' also become jokers.
\begin{proposition}\label{prop:jokertransfer}
Let $h,m\in\NN$ and suppose that $h<m<j\le k-2$ and $j+m-h\le k-1$. Let $M_1$ and $M_2$ be $m$-sets such that $|M_1\cap M_2|=h$. If $M_1^*\subset \cA_t$ then $M_2^*\subset \cA_{t+1}$.
\end{proposition}
\begin{proof}
Let $J$ be any $j$-set containing $M_2$, and we aim to show that $J \in \cA_{t+1}$.
Let us define $U\coloneqq J \cup M_1$, and observe that $|U| \le j+m-h \le k-1$.
We can therefore choose a $(k-|U|-1)$-set $S'\subset [n]\backslash U$ and distinct vertices $x_1,\ldots,x_r\in [n]\backslash(U\cup S')$,
and set $S_i\coloneqq S'\cup\{x_i\}$ for $i\in [r]$. Then setting $K_i\coloneqq U \cup S_i$ for $i\in [r]$ (see Figure~\ref{fig:Kiconstruction}),
these are distinct $k$-sets since each contains precisely one of the $x_i$.

\begin{center}
\begin{figure}[h]
\begin{tikzpicture}[scale=1]
\vertex{.005cm}{black} at (2,3) (23) {};
\vertex{.005cm}{black} at (3,3) (33) {};
\vertex{.005cm}{black} at (4,3) (43) {};
\vertex{.005cm}{black} at (5,3) (53) {};
\vertex{.005cm}{black} at (6,3) (63) {};
\vertex{.005cm}{black} at (4.4,2.15) (42) {};
\vertex{.005cm}{black} at (5.6,2.15) (62) {};
\vertex{.005cm}{black} at (5,2.15) (52) {};
\vertex{.005cm}{black} at (3.5,4.5) (35) {};
\vertex{.005cm}{black} at (4.5,4.5) (45) {};
\vertex{.005cm}{black} at (4,5.7) (46) {};
\foreach \x in {4,...,10}
\svertex{.0005cm}{black} at (\x/3,5.7) {};
\foreach \x in {14,...,20}
\svertex{.0005cm}{black} at (\x/3,5.7) {};
\vertex{.005cm}{black} at (4/3,5.7) {};
\node at (4/3,5.4) {\small \color{black} $x_1$};
\vertex{.005cm}{black} at (20/3,5.7) {};
\node at (20/3,5.4) {\small \color{black} $x_r$};
\draw (4,4.5) ellipse (1cm and 0.3cm);
\draw (3,3) ellipse (1.5cm and 0.4cm);
\draw (5,3) ellipse (1.5cm and 0.4cm);
\drawFourHyperedge{62}{42}{43}{63}{black}{.2cm};
\draw [rounded corners=13pt]
(1,2)--(1,5.2)--(3.2,5.2)--(4,6.2)--(4.8,5.2)--(7,5.2)--(7,1.5)--(1,1.5)--(1,2);
\node at (4,5.4) {\small \color{black} $x_i$};
\node at (2.5,4.5) {\small \color{black} $S'$};
\node at (1.5,3.4) {\small \color{black} $M_1$};
\node at (6.5,3.4) {\small \color{black} $M_2$};
\node at (6,2) {\small \color{black} $J$};
\node at (7.3,3.8) {\small \color{black} $K_i$};
\node at (0.3,2.7) {\small \color{black} $U$};
\draw [decorate,ultra thick,
decoration = {calligraphic brace,
raise=5pt,
amplitude=8pt}] (1,1.8) -- (1,3.6);
\end{tikzpicture}
\caption{
The construction of $K_i$ with $m=3$, $h=1$, $j=6$ and $k=11$.
In this example $J \cap(M_1\setminus M_2) = \emptyset$, although this need not be true in general.
}\label{fig:Kiconstruction}
\end{figure}
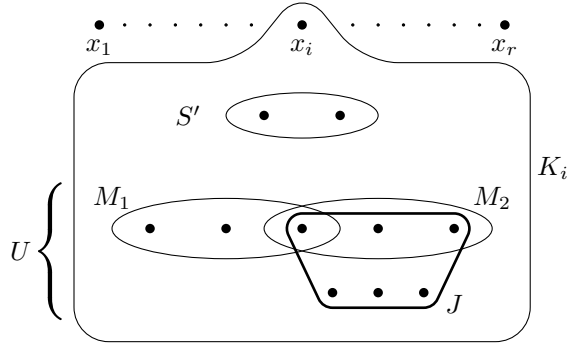
\end{center}

Further, for each $i \in [r]$, let $J_i$ be a $j$-set consisting of $M_1\cup\{x_i\}$ and
$j-m-1$ further vertices of $U$ (chosen arbitrarily). Note that $|U| \ge |J| = j$, so there
are certainly enough vertices of $U$ available.

Now since $M_1 \subset J_i$ and $M_1^* \subset \cA_t$, we have $J_i \in \cA_t$ for each $i\in [r]$; furthermore, the $J_i$ are all distinct because 
they each contain the corresponding $x_i$; and finally $(J \cup J_i) \subset K_i$. It follows that $J \in \cA_{t+1}$,
and since $J$ was an arbitrary $j$-set containing $M_2$, the claim follows.
\end{proof}

We remark that Proposition~\ref{prop:jokertransfer} does not hold for the tight case $j=k-1$: Even if $h = m-1$, i.e., $M_1$ and $M_2$ are as close 
together as they can be without being identical, for a $j$-set $J$ containing $M_2$ but not $M_1$, we have $|M_1 \cup J| = j+1 = k$, and therefore we can 
find only one edge $K_1 = M_1 \cup J$ rather than the $r$ distinct edges we would require to infect $J$. This is the fundamental reason why the tight 
case is different and is reflected in the condition $j+m-h\le k-1$, which can only be satisfied with the appropriate choice of $m$ and $h$ if $j\le k-2$. 

As a corollary of Proposition~\ref{prop:jokertransfer}, we observe that if we have a joker $m$-set at time $t$ we will have percolation after at most $m$ further steps. 
\begin{corollary}\label{cor:jokertransfer}
Let $m\in\NN$ and suppose that $m<j\le k-2$. Let $M\in\binom{[n]}{m}$. If $M^* \subset \cA_t$, then $\cA_{t+m} = \binom{[n]}{j}$.
\end{corollary}
\begin{proof}
Consider any $m$-set $\tilde{M}$ and set $h'\coloneqq|M\cap \tilde{M}|$. Fix a sequence of $m$-sets 
$M=M_1,M_2,\ldots,M_{1+m-h'}=\tilde{M}$ with $|M_i\cap M_{i+1}|=m-1$ for $i\in [m-h']$.
We claim that $M_i^*\subset \cA_{t+i-1}$ for all $i\in[1+m-h']$ and prove this statement by induction on $i$.
The base case $M_1^*=M^*\subset \cA_t$ holds by assumption. Suppose that the statement holds for $i$, so $M_i^*\subset \cA_{t+i-1}$.
Setting $h\coloneqq|M_i\cap M_{i+1}|=m-1$ we have $h<m<j\le k-2$ and $j+m-h=j+1\le k-1$,
and by Proposition~\ref{prop:jokertransfer} we deduce that $M_{i+1}^*\subset \cA_{(t+i-1)+1}$, which proves the induction hypothesis.

In particular, we have $\tilde{M}^*\subset \cA_{t+m-h'}\subset \cA_{t+m}$.
Since $\tilde{M}$ was chosen arbitrarily and every $j$-set contains such an $m$-set $\tilde{M}$, we have $\cA_{t+m}=\binom{[n]}{j}$.
\end{proof}

We can now complete the proof of the first statement of Theorem~\ref{thm:maintheorem}.

\begin{lemma}\label{lem:mostcases}
Let $m\in\NN$ and suppose that $2j-k\le m<j\le k-2$. Then an $(m,j)$-star $\cS$ of size~$r$ percolates in $\cK_n^k$.
In particular, if $j \le k-2$ we have $\ell_n(k,j,r)=r$.
\end{lemma}
\begin{proof}
Let $M$ be the centre of the star $\cS$. Setting $\cA_0=\cS$, Proposition~\ref{prop:getfullstar} implies that $M^*\subset \cA_{1}$. By Corollary~\ref{cor:jokertransfer} 
we have $\cA_{1+m} = \binom{[n]}{j}$ and so $S$ is a contagious starting configuration which consists of $r$ distinct $j$-sets.
This shows that $\ell_n(k,j,r)\le r$ provided we can find an integer $m$ satisfying $2j-k\le m <j \le k-2$, and this is
always possible (by setting $m=j-1$) if $j \le k-2$; the corresponding lower bound is trivial.
\end{proof}

\section{Preliminaries for the tight case}\label{sec:prelimtight}

In the remainder of the paper, our focus will be on the tight case, namely $j=k-1$.
In view of Theorem~\ref{thm:maintheorem}, we will further distinguish the subcases $k=3$ and $k\ge 4$.
First, in this section we present some auxiliary results and terminology which will be needed for both subcases.

\begin{definition}[Extension of a $j$-set]\label{def:extension}
Given a $k$-uniform hypergraph $\cH=(V,\cE)$ and a $j$-configuration $\cJ\subset\binom{V}{j}$,
we say a vertex $v\in V$ is an \emph{extension for a $j$-set $J\in\binom{V}{j}$} 
if there exists $J'\in \cJ\backslash J$ with $J'\subset (J\cup\{v\})\in \cE$. We define the \emph{extension set of $J$} as 
\[\extset_{\cJ}(J)=\extset_{\cJ,\cH}(J)\coloneqq\{v\in V:v \text{ is an extension for }J \}.\]
\end{definition}

The following is a useful equivalent formulation of Definition~\ref{def:infection} in the tight case.
\begin{claim}\label{claim:whenpairgetsinfected}
Suppose that $1\le j=k-1$. Then in the $(r,\cA_0)$-infection process, for any $j$-set $J\not\in \cA_t$ we have
\[J\in \cA_{t+1}\Longleftrightarrow |\extset_{\cA_t}(J)|\ge r.\]
\end{claim}
\begin{proof}
We have $J\in \cA_{t+1}$ iff there exist distinct edges $K_1,\ldots,K_r$ and distinct $j$-sets $J_1,\ldots,J_r\in \cA_t$ with $(J_i\cup J)\subset K_i$ 
for all $i\in[r]$. Since $k=j+1$ this holds iff there are at least $r$ distinct vertices $w_1,\ldots,w_r$ with $K_i\coloneqq J\cup\{w_i\}$ and $J_i=\{w_i\}\cup S$ 
where $S\subset J$ with $|S|=j-1$. By Definition~\ref{def:extension} this is equivalent to $w_i\in \extset_{\cA_t}(J)$ for every $i$ 
and therefore also to $|\extset_{\cA_t}(J)|\ge r$.
\end{proof}

We observe that having $r$ jokers is sufficient for percolation in the tight case. The proof follows immediately from 
Lemma~\ref{lem:reductionlemma} with $s=r$ or can be easily proved directly.
\begin{claim}\label{claim:sufficientforperc}
Suppose that $2\le j=k-1$. Consider the $(r,\cA_0)$-infection process and suppose that there exists $t\in\NN$ such that 
we have $r$ distinct jokers for $\cA_t$. Then $\cA_{t+1}=\binom{[n]}{j}$.\qed
\end{claim}

\section{Special tight case: $k=3,\, j=2$}\label{sec:specialtight}
We proceed with the proof of Theorem~\ref{thm:maintheorem}\ref{item:exactj}. We prove the upper and lower bounds on $\ell_n(3,2,r)$ separately.
\subsection{Lower bound}

We set $W_0 \coloneqq \emptyset$ and recursively define pairs $P_1,P_2,\ldots$, and sets 
 $W_1,W_2,\ldots \subset [n]$,
where $P_i$ is the first pair of $\binom{[n]\setminus W_{i-1}}{2}$ to become infected and
$W_i := \bigcup_{i'=1}^i V(P_{i'})$; if there is more than one possible choice for $P_i$, we pick one arbitrarily.

Note that $P_1$ is simply the first pair to become infected, $P_2$ is the first pair not incident to $P_1$ to become infected,
and so on, so the $P_i$ form a matching. Clearly the $P_i$ are only well-defined for as long as such a pair exists.
The following observation is crucial.

\begin{proposition}\label{prop:extensions}
For any integer $i\ge 0$, if $P_i$ exists we have $|\extset_{\cA_0}(P_i) \setminus W_{i-1}| \ge r-2(i-1)$.
\end{proposition}

\begin{proof}
Suppose that $P_i \in \cN_h$ for some integer $h$, i.e., it becomes infected in step~$h$ of the process.
Then by Claim~\ref{claim:whenpairgetsinfected} we have $|\extset_{\cA_{h-1}}(P_i)| \ge r$.
On the other hand, by the definition of $P_i$, all pairs of $\cA_{h-1} \setminus \cA_0$ are incident
to $W_{i-1}$ (since they become infected before $P_i$, and could otherwise have been chosen as $P_i$).
It follows that
$(\extset_{\cA_{h-1}}(P_i)\setminus W_{i-1}) \subset  (\extset_{\cA_0}(P_i) \setminus W_{i-1})$, and therefore
$$
|\extset_{\cA_0}(P_i)\setminus W_{i-1}| \ge |\extset_{\cA_{h-1}}(P_i)| - |W_{i-1}| \ge r -2(i-1),
$$
as required.
\end{proof}

We can now deduce the following.

\begin{lemma}\label{lem:lowerboundspecialtight}
If a $2$-configuration $\cA_0$ percolates, then $|\cA_0|\ge\frac{1}{4}\left((r+1)^2-\mathbbm{1}_{r\in 2\NN}\right)$.
In particular, $\ell_n(3,2,r)\ge \frac{1}{4}\left((r+1)^2-\mathbbm{1}_{r\in 2\NN}\right)$.
\end{lemma}

\begin{proof}
Let $s$ be the number of $P_i$ which exist, i.e., $P_1,\ldots,P_s$ are well-defined, but subsequently no pair in $[n]\setminus W_s$ becomes infected.
We let $z\coloneqq\lceil\frac{r}{2}\rceil$ and distinguish two cases.

\textbf{Case 1: $s<z$.}\\
By the definition of $s$, no pair of $\cA_\tau \setminus \cA_0$ lies in $[n]\setminus W_s$.
Therefore if $\cA_0$ percolates, we have $\binom{[n]\setminus W_s}{2} \subset \cA_0$,
and therefore
$$
|\cA_0| \ge \binom{n-|W_s|}{2} = \binom{n-2s}{2} \ge \binom{n-r}{2}
\ge\frac{(r+1)^2}{4}\ge\frac{1}{4}\left((r+1)^2-\mathbbm{1}_{r\in 2\NN}\right),
$$
where we used the fact that $n\ge n_0 \ge 2r+1$.

\textbf{Case 2: $s\ge z$.}\\
We first claim that
$$
|\cA_0| 
 \ge \sum_{i=1}^z|\extset_{\cA_0}(P_i)\setminus W_{i-1}|.
$$
To see this, observe that for each $1\le i \le z$, we can pick
$|\extset_{\cA_0}(P_i)\setminus W_{i-1}|$ distinct pairs of $\cA_0$ which are incident to~$P_i$,
but not incident to~$W_{i-1}$ (importantly, by construction $P_i$ is itself not incident to~$W_{i-1}$).
Then since $P_i \subset W_i$, the pairs that we chose in this way are necessarily all distinct even
over different choices of $i$.

We now simply apply Proposition~\ref{prop:extensions} to deduce that 
$$
|\cA_0| \ge \sum_{i=1}^z |\extset_{\cA_0}(P_i)\setminus W_{i-1}|
\ge \sum_{i=1}^{z}(r-2(i-1))
\ge\frac{1}{4}\left((r+1)^2-\mathbbm{1}_{r\in 2\NN}\right),
$$
as required.
\end{proof}

\subsection{Upper bound}
In order to prove an upper bound on $\ell_n(3,2,r)$ we will provide a construction of the appropriate size that percolates.
\begin{definition}\label{def:optimalstars}
For $r \in \NN$ we define $\stars_r$ to be a set of $r$ vertex-disjoint $(1,2)$-stars of sizes $1,1,2,2,\ldots,\lceil\frac{r-1}{2}\rceil,\lceil\frac{r}{2}\rceil$.
\end{definition}

\begin{center}
\begin{figure}[h]
\captionsetup{justification=centering,margin=2cm}
\begin{tikzpicture}[scale=1.4]
  \vertex{.03cm}{black} at (0,0) (00) {};
  \vertex{.03cm}{black} at (0,1) (01) {};
  \vertex{.03cm}{black} at (1,0) (10) {};
  \vertex{.03cm}{black} at (1,1) (11) {};
  \vertex{.03cm}{black} at (2,0) (20) {};
  \vertex{.03cm}{black} at (1.8,1) (181) {};
  \vertex{.03cm}{black} at (2.2,1) (221) {};
  \vertex{.03cm}{black} at (3,0) (30) {};
  \vertex{.03cm}{black} at (2.8,1) (281) {};
  \vertex{.03cm}{black} at (3.2,1) (321) {};
  \vertex{.03cm}{black} at (4,0) (40) {};
  \vertex{.03cm}{black} at (4,1) (41) {};
  \vertex{.03cm}{black} at (3.7,1) (381) {};
  \vertex{.03cm}{black} at (4.3,1) (421) {};
  \drawPair{00}{01}{black};
  \drawPair{10}{11}{black};
  \drawPair{20}{181}{black};
  \drawPair{20}{221}{black};
  \drawPair{30}{281}{black};
  \drawPair{30}{321}{black};
  \drawPair{40}{381}{black};
  \drawPair{40}{421}{black};
  \drawPair{40}{41}{black};
  \node at (0,-0.35) {\small \color{black} $v_1$};
  \node at (1,-0.35) {\small \color{black} $v_2$};
  \node at (2,-0.35) {\small \color{black} $v_3$};
  \node at (3,-0.35) {\small \color{black} $v_4$};
  \node at (4,-0.35) {\small \color{black} $v_5$};
\end{tikzpicture}
\caption{The $2$-configuration $\stars_5$}
\end{figure}
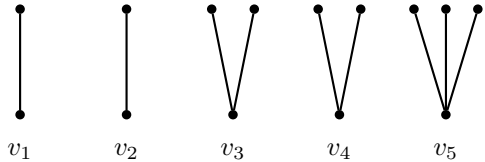
\end{center}
Whenever we have such a configuration $\stars_r$ we will use $v_1,\ldots,v_r$ to denote the centres of the stars of sizes
$1,1,2,2,\ldots,\lceil\frac{r-1}{2}\rceil,\lceil\frac{r}{2}\rceil$, respectively. Note that the choice of the centres is not uniquely determined,
but we select an appropriate choice arbitrarily. We note that in fact $v_i=v_i(\stars_r)$ for all $i\in [r]$.

For vertices $u,v$, to ease notation we will use $uv$ as a shorthand for $\{u,v\}$.

\begin{proposition}\label{prop:lastvertexgetsjoker}
Suppose $\cA_0 \supset \stars_r$. Then $\{v_r\}^* \subset \cA_{r}$, i.e., $v_r$ becomes a joker after at most $r$ steps.
\end{proposition}
\begin{proof}
We first claim that after step $i-1$ of the $(r,\cA_0)$-infection process we have that $v_r v_{i'}\in \cA_{i-1}$ for all $r-i<i'< r$
and we prove this claim by induction on $i\in[r]$. For $i=1$ the statement is empty and therefore holds trivially.
Suppose the claim is true for $i\ge 1$ and we aim to prove the same statement for $i+1$, for which it suffices to show that $v_{r-i}v_r \in \cA_i$.
For $q\in [r]$,
let $L_q$ denote the set of leaves of the star $\cS_q\subset\stars_r$ with centre~$v_q$. Then
\[
\extset_{\cA_{i-1}}(v_{r-i}v_r)\supset \left(L_{r-i}\cup L_r\cup \{v_{r-i+1},\ldots,v_{r-1}\}\right).
\]
Since $|L_q|=\lceil\frac{q}{2}\rceil$ for $q\in [r]$, we have 
\[
|\extset_{\cA_{i-1}}(v_{r-i}v_r)|
\ge \left\lceil\frac{r-i}{2}\right\rceil+\left\lceil\frac{r}{2}\right\rceil+(r-1)-(r-i)
=r+\frac{i}{2}-1+\frac{1}{2}\left(\mathbbm{1}_{r-i\in 2\NN-1}+\mathbbm{1}_{r\in 2\NN-1}\right)\ge r.
\]
It follows by Claim~\ref{claim:whenpairgetsinfected} that $v_{r-i}v_r\in \cA_{i}$ which proves the statement of our claim for $i+1$,
and thus we have proved the claim by induction.

In particular, the case $i=r$ implies that $d_{\cA_{r-1}}(v_r)\ge |\{v_1,\ldots,v_{r-1}\} \cup L_r| =  r-1+\lceil\frac{r}{2}\rceil\ge r$, 
and by Proposition~\ref{prop:getfullstar} we conclude that $\{v_r\}^*\subset \cA_{r}$.
\end{proof}

We can use Proposition~\ref{prop:lastvertexgetsjoker} to show that all $v_i\in\stars_r$ will eventually become jokers in the 
infection process starting from $\stars_r$.
\begin{corollary}\label{cor:alljokers}
Suppose that $\cA_0\supset\stars_r$ and let $R\coloneqq\{v_1,\ldots,v_r\}$. Then there exists a $T\in\NN$ such that $\bigcup_{s \in [r]}v_s^*\subset \cA_T$, 
i.e., $v_1,\ldots,v_r$ will become jokers by some time $T$ in the $(r,\cA_0)$-infection process.
\end{corollary}
\begin{proof}
We prove the statement by induction on $r$. In the case $r=0$, the statement is empty and therefore holds trivially,
so suppose the statement is true for $r-1$. By Proposition~\ref{prop:lastvertexgetsjoker}
we have $v_r^*\subset \cA_r$, i.e., $v_r$ becomes a joker after at most $r$ steps. By Lemma~\ref{lem:reductionlemma} we can subsequently analyse
the equivalent $(r-1,\cA_r-\{v_r\})$-infection process on $\cK_n^3-\{v_r\}$ (the reduction gives an equivalence since we are in the tight case).
Since $\stars_{r-1} \subset (\cA_0 -\{v_r\}) \subset (\cA_r - \{v_r\})$,
we apply the induction hypothesis and deduce that there exists an integer $t$ such that $v_1,\ldots,v_{r-1}$
become jokers for $\cA_{r+t}$ in $[n]\setminus\{v_r\}$. Since $v_r$ is a joker in $[n]$ clearly $v_1,\ldots,v_{r-1}$ also become jokers in $[n]$
after at most $r+t$ steps.
\end{proof}

We can now deduce Theorem~\ref{thm:maintheorem}\ref{item:exactj}. 

\begin{lemma}\label{lem:upperboundspecialtight}
Suppose $2=j=k-1$. The $2$-configuration $\stars_r$ is contagious with infection threshold $r$.
In particular,
\[\ell_n(3,2,r)= |\stars_r|=\frac{1}{4}\left((r+1)^2-\mathbbm{1}_{r\in 2\NN}\right).\]
\end{lemma}
\begin{proof}
Set $\cA_0\coloneqq \stars_r$ and apply Corollary~\ref{cor:alljokers} to deduce that $v_1,\ldots,v_r$ are jokers after some time~$T$. By Claim~\ref{claim:sufficientforperc} we then have that $\cA_{T+1}=\binom{[n]}{j}$. 

It follows that  
\[\ell_n(3,2,r)\le|\stars_r|=1+1+\ldots+\left\lceil \frac{r-1}{2}\right\rceil+\left\lceil \frac{r}{2}\right\rceil=\frac{1}{4}\left((r+1)^2-\mathbbm{1}_{r\in 2\NN}\right).\]
The corresponding lower bound is given by Lemma~\ref{lem:lowerboundspecialtight}.
\end{proof}

\section{General tight case: $k\ge 4$}\label{sec:generaltight}

We now proceed with the proof of Theorem~\ref{thm:maintheorem}\ref{item:largej}.
We first need to introduce some new notation which generalises the notion of an $i$-configuration.

\begin{definition}
Let $m,i,s\in\NN$. An $(m;i,s)$-configuration is an $i$-configuration which is contagious with infection threshold $s$ in $\cK_m^{i+1}$.
\end{definition}
Our next definition describes a way of augmenting a given configuration by a vertex.
\begin{definition}\label{def:augmentation}
Given an $(i-1)$-configuration $\cB$ and some vertex $v\not\in V[\cB]$, the \emph{$v$-augmented configuration}
$\cB_v$ is the $i$-configuration that results from $\cB$ by adding $v$ to each $(i-1)$-set of $\cB$ to create an $i$-set.
We call $v$ the \emph{master vertex} of the $v$-augmented configuration. 
\end{definition}

\begin{center}
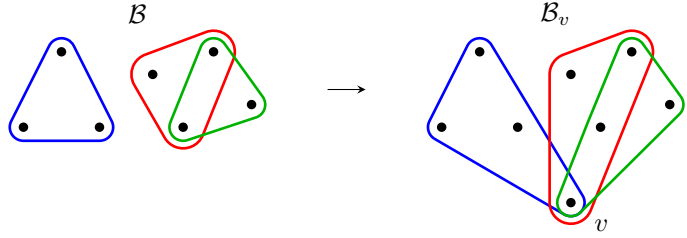
\begin{figure}[h]
\captionsetup{justification=centering,margin=2cm}
\begin{tikzpicture}[scale=1]
\vertex{.03cm}{black} at (0,0) (00) {};
\vertex{.03cm}{black} at (1,0) (10) {};
\vertex{.03cm}{black} at (0.5,1) (051) {};
\vertex{.03cm}{black} at (1.7,0.7) (305) {};
\vertex{.03cm}{black} at (2.5,1) (351) {};
\vertex{.03cm}{black} at (2.1,0) (350) {};
\vertex{.03cm}{black} at (3,0.3) (405) {};
\draw [-stealth](4,0.5) -- (4.5,0.5);

\vertex{.03cm}{black} at (5.5,0) (550) {};
\vertex{.03cm}{black} at (6.5,0) (650) {};
\vertex{.03cm}{black} at (6,1) (61) {};
\vertex{.03cm}{black} at (7.2,0.7) (8505) {};
\vertex{.03cm}{black} at (8,1) (91) {};
\vertex{.03cm}{black} at (7.59,0) (90) {};
\vertex{.03cm}{black} at (8.5,0.3) (9505) {};
\vertex{.03cm}{black} at (7.2,-1) (v) {};
\drawThreeHyperedge{051}{10}{00}{blue}{.18cm};
\drawThreeHyperedge{351}{350}{305}{red}{.28cm};
\drawThreeHyperedge{350}{351}{405}{darkgreen}{.18cm};

\drawThreeHyperedge{550}{61}{v}{blue}{.18cm};
\drawThreeHyperedge{91}{v}{8505}{red}{.28cm};
\drawFourHyperedge{v}{90}{91}{9505}{darkgreen}{.18cm};

  \node at (7.6,-1.3) {\small \color{black} $v$};
  \node at (1.5,1.5) {\small \color{black} $\cB$};
  \node at (7,1.5) {\small \color{black} $\cB_v$};

\end{tikzpicture}
\caption{The augmentation of a $3$-configuration $\cB$ by $v$.}
\end{figure}
\end{center}

Given a set $S$ of vertices, let $\configs_{i,S}$ be the set of $i$-configurations disjoint from $S$,
by which we mean every $i$-set of the $i$-configuration is disjoint from $S$.
If $S=\{v\}$ then we define $\configs_{i,v}\coloneqq\configs_{i,\{v\}}$ for ease of notation.
Observe that the mapping $\cB \mapsto \cB_v$ defines an \emph{augmentation function} $\configs_{i-1,v} \to \configs_{i,\emptyset}$.
Note that for $\cA,\cB\in\configs_{i-1,v}$ we have $(\cA\cup \cB)_v=\cA_v\cup \cB_v$.

Our next lemma couples two infection processes which differ in the augmentation of the initially infected sets.

\begin{lemma}\label{lem:inclusionpropf}
Let $v$ be a vertex and let $\cC_0\in\configs_{j-1,v}$ and $\cC_0'\in\configs_{j,\emptyset}$. Let $\cC_t$ be the set of infected $(j-1)$-sets 
after $t$ steps in the $(r,\cC_0)$-infection process in $\cK_n^{k-1}- \{v\}$
and let $\cC_t'$ be defined similarly for the $(r,\cC_0')$-infection process in $\cK_n^{k}$. 
If $(\cC_0)_v\subset \cC_0'$ then $(\cC_t)_v\subset \cC_t'$ for all $t\in\NN$.
\end{lemma}
\begin{proof}
We prove the statement by induction on $t$. The case $t=0$ follows by assumption.
Suppose that $(\cC_{t-1})_v\subset \cC_{t-1}'$ for some $t\in\NN$ and let $\cN_t\coloneqq \cC_{t}\setminus \cC_{t-1}$
and $\cN_t'\coloneqq \cC_{t}'\setminus \cC_{t-1}'$.
We need to show that $(\cC_{t-1}\cup \cN_t)_v\subset \cC_{t}'$.
It suffices to show that $(\cN_t)_v\subset (\cC_{t-1}'\cup \cN_t')$ because then by the induction hypothesis 
\[
(\cC_t)_v=(\cC_{t-1}\cup \cN_t)_v=((\cC_{t-1})_v\cup (\cN_t)_v)\subset (\cC_{t-1}'\cup \cN_{t}')=\cC_{t}'.
\]
Let $J\in \cN_t$, and we aim to show that $J \cup \{v\} \in \cC_{t-1}'\cup \cN_t'$.
Assume that $J\cup \{v\} \notin \cC_{t-1}'$, otherwise we are done,
and let $K_1,\ldots,K_r \in \binom{[n]\setminus \{v\}}{k-1}$ and $J_1,\ldots,J_r \in \cC_{t-1}$
be the $(k-1)$-sets and $(j-1)$-sets which infected $J$ in step $t$.
Let $J_i' = J_i\cup \{v\}$ and $K_i' = K_i \cup \{v\}$. It is easy to check that these are distinct $k$-sets and $j$-sets (since the $K_i$ and $J_i$ are distinct).
Since $J_i \in \cC_{t-1}$ we have $J_i' \in \cC_{t-1}'$ by the induction hypothesis, so $J\cup\{v\}\in \cN_t'$.
Since $J$ was chosen arbitrarily, this concludes the proof. 
\end{proof}

\begin{corollary}\label{cor:lastonejoker}
Let $\cB\in\configs_{j-1,v}$ be an $(n-1;j-1,r)$-configuration.
If $\cB_v\subset \cA_0$ then there exists a $t\in\NN$ such that $v^*\subset \cA_t$ in the $(r,\cA_0)$-infection process.
\end{corollary}
\begin{proof}
By assumption the configuration $\cB$ percolates in $[n]\backslash\{v\}$, hence there is a $t$ at which every $(j-1)$-set is infected 
in $[n]\backslash\{v\}$ and the statement follows by Lemma~\ref{lem:inclusionpropf}.
\end{proof}

\begin{proposition}\label{prop:mastertojoker}
For $s\in[r]$ suppose that we have $(n-r;j-1,s)$-configurations $\cB_s$ and a set of $r$ distinct vertices $R=\{v_1,\ldots,v_r\}$
which is disjoint from every $\cB_s$. Let $\cA\coloneqq\bigcup\limits_{s\in[r]}(\cB_s)_{v_s}$ and suppose $\cA \subset \cA_0$.
Then there exists $T\in\NN$ such that $\bigcup_{s\in [r]}v_s^* \subset \cA_T$, i.e., $v_1,\ldots,v_r$ are jokers at time $T$ in the $(r,\cA_0)$-infection process.
\end{proposition}
\begin{proof}
We prove the statement by induction on $r$. The proof is very similar to the proof of Corollary~\ref{cor:alljokers} and therefore we will only sketch the proof.
The case $r=0$ holds trivially since it is an empty statement, so suppose the statement holds for $r-1$.
By Corollary~\ref{cor:lastonejoker} there exists $t_0$ with $\{v_r\}^*\subset \cA_{t_0}$. We observe that 
$
\bigcup\limits_{i=1}^{r-1}(\cB_i)_{v_i}\subset (\cA-\{v_r\}) \subset (\cA_{t_0}-\{v_r\})
$
and so by the induction hypothesis the vertices in $R\setminus \{v_r\}$ become jokers by some time $t_1$
in the $(r-1,\cA_{t_0}-\{v_r\})$-infection process in $\cK_n^k -\{v_r\}$.
By Lemma~\ref{lem:reductionlemma} (and the fact that $v_r$ is also a joker),
these vertices are also jokers in $\cA_{t_0+t_1}$ in the $(r,\cA_0)$-infection process in $\cK_n^k$, as required.
\end{proof}

\begin{corollary}\label{cor:candidateconstructionpercolates}
Let $\cA$ be as in Proposition~\ref{prop:mastertojoker} and suppose that $\cA\subset \cA_0$. Then the $(r,\cA_0)$-infection process percolates.
\end{corollary}
\begin{proof}
By Proposition~\ref{prop:mastertojoker} there exists a time $T$ such that $v_1,\ldots,v_r$ are jokers at time $T$. By Claim~\ref{claim:sufficientforperc} the statement follows. 
\end{proof}

\begin{corollary}\label{lem:recursionformula}
If $3\le j=k-1$, then 
\[\ell_n(k,j,r)\le\sum_{i=1}^{r}\ell_n(k-1,j-1,i).\]
\end{corollary}
\begin{proof}
Let $\cB_s$ and $\cA$ be as in Proposition~\ref{prop:mastertojoker}, and additionally assume that the $\cB_s$ are \emph{minimal} $(n;j-1,i)$-configurations,
so $|\cB_s|=\ell_n(k-1,j-1,s)$. By Corollary~\ref{cor:candidateconstructionpercolates} we have that $\cA$ percolates and since
\[
\ell_n(k,j,r)\le|\cA|=\sum_{s=1}^{r}|\cB_s|=\sum_{s=1}^{r}\ell_n(k-1,j-1,s),
\]
the statement follows.
\end{proof}

\section{Proof of Theorem~\ref{thm:maintheorem}}
\label{sec:proofofmainthm}

We observe that we have now proved
all the individual statements of our main theorem.

\begin{proof}[Proof of Theorem~\ref{thm:maintheorem}]
Statement~\ref{item:smallj} is simply the ``in particular'' part of Lemma~\ref{lem:mostcases} and
similarly~\ref{item:exactj} is contained in Lemma~\ref{lem:upperboundspecialtight}.
Finally~\ref{item:largej} is precisely the statement of Corollary~\ref{lem:recursionformula}.
\end{proof}

\section{Concluding Remarks}\label{sec:remarks}
If we iterate the sum in Theorem~\ref{thm:maintheorem}\ref{item:largej} $k-3$ times and then replace every summand with the 
expression given in Theorem~\ref{thm:maintheorem}\ref{item:exactj} we obtain 
\[\ell_n(k,j,r)\le\frac{2r^2+r(5k-11)-17(k-1)+4k^2}{4(k-1)!}(r)_{k-3}-\frac{1}{4}\sum_{i=1}^{\left\lfloor\frac{r}{2}\right\rfloor}\binom{r+k-3-2i}{k-4}.\] 
Note that this expression is a polynomial in $r$ of degree $k-1$. 
However, this bound is far from best possible in general.

\begin{claim}\label{claim:counterexample}
If $j=k-1$ and $\cA_0 \subset \binom{[n]}{j}$ is isomorphic to $\cK_{j+r-1}^j$,
then the $(r,\cA_0)$-infection process in $\cK_n^k$ percolates. In particular, $\ell_n(j+1,j,r) \le \binom{j+r-1}{j}$.
\end{claim}

\begin{proof}[Proof sketch]
It is easy to show inductively that for $i \in [j]$, all $(j-i)$-sets within $V(\cA_0)$ are contained in
at least $j+r-1 - (j-i) \ge r$ joker $(j-i+1)$-sets for $\cA_{i-1}$, and therefore become joker $(j-i)$-sets for $\cA_i$.
The case $i=j$ states that the empty set is a joker, i.e., all $j$-sets are infected.
\end{proof}

In particular, observe that if $k-1=j \gg r$, then $\ell_n(k,j,r) \le \binom{j+r-1}{j} \le j^r \ll r^{k-1}$.
This shows that in general $\ell_n(k,k-1,r)$ cannot be a polynomial in $r$ of degree $k-1$.
Thus Theorem~\ref{thm:maintheorem}\ref{item:largej} represents the first step
towards determining $\ell_n(k,k-1,r)$, and
it would be interesting
to investigate the correct behaviour of this function.

We conclude with a concrete example in the case where $k=4$ and $j=3$, and where $r=3$.
We have by Theorem~\ref{thm:maintheorem} that $\ell_n(4,3,3)\le\ell_n(3,2,1)+\ell_n(3,2,2)+\ell_n(3,2,3)=7$
(while Claim~\ref{claim:counterexample} only gives $\ell_n(4,3,3)\le 10$).
However, this is not the correct value, as the example in Figure~\ref{fig:counterexample} shows.
\vspace{0.2cm}
\begin{center}
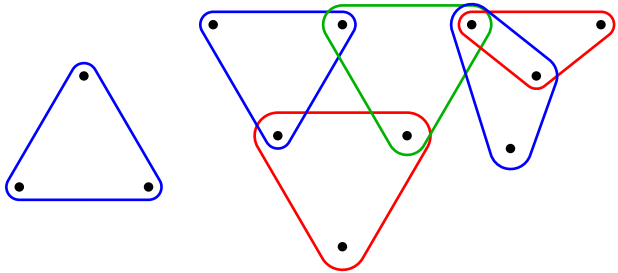
\begin{figure}[h]
\captionsetup{justification=centering,margin=2cm}
\begin{tikzpicture}[scale=1.7]

\vertex{.005cm}{black} at (2.5,-0.866) (25u) {};
\vertex{.005cm}{black} at (0.5,0.466) (05m) {};
\vertex{.005cm}{black} at (1.5,0.866) (15m) {};
\vertex{.005cm}{black} at (2.5,0.866) (v2) {};
\vertex{.005cm}{black} at (3.5,0.866) (35m) {};
\vertex{.005cm}{black} at (4.5,0.866) (y) {};
\vertex{.005cm}{black} at (0,-0.4) (00) {};
\vertex{.005cm}{black} at (1,-0.4) (10) {};
\vertex{.005cm}{black} at (2,0) (v1) {};
\vertex{.005cm}{black} at (3,0) (v3) {};
\vertex{.005cm}{black} at (4,0.466) (v4) {};
\vertex{.005cm}{black} at (3.8,-0.1) (v5) {};

\drawThreeHyperedge{v1}{v3}{25u}{red}{.18cm};
\drawThreeHyperedge{15m}{v2}{v1}{blue}{.1cm};
\drawThreeHyperedge{v2}{35m}{v3}{darkgreen}{.15cm};

\drawThreeHyperedge{35m}{y}{v4}{red}{.1cm};

\drawThreeHyperedge{00}{05m}{10}{blue}{.1cm};
\drawThreeHyperedge{35m}{v4}{v5}{blue}{.16cm};
    
%
 \end{tikzpicture}
 \caption{A $3$-configuration with $6$ edges.}\label{fig:counterexample}
 \end{figure}
\end{center}

One can check that this $3$-configuration does indeed percolate in $\cK_n^4$ with infection threshold $3$,
and also that there is no percolating configuration containing only $5$ triples, so in fact $\ell_n(4,3,3)=6$ -- we omit the proof.

\bibliographystyle{plain}
\bibliography{References}
\end{document}